\documentclass{amsart}
\usepackage{amsmath,amsfonts,amsthm,amstext,amscd}
\usepackage[latin1]{inputenc}
\usepackage[dvips]{graphicx}
\usepackage{array}
\usepackage{hyperref}
\usepackage{newlfont}
\def\rm#1{\mathrm{#1}}
\def\cal#1{\mathcal{#1}}

\def\lr#1{\langle #1\rangle}
\usepackage[all]{xy}

 \newtheorem{thm}{Theorem}[section]
 
 \newtheorem{lem}[thm]{Lemma}
 \newtheorem{prop}[thm]{Proposition}
 \theoremstyle{definition}
 
 \theoremstyle{remark}

 \numberwithin{equation}{section}

\begin{document}

%
%
%
%
%
%
%
%
%

\title[Curvature of Pull-back-bundles]
 {An obstructions to Non-negative curvature on Bundles with Pull-back connections}

\author[Speran\c ca]{L. D. Speran\c ca}

\address{L. D. Sperança\\IMECC - Unicamp,
Rua S\'ergio Buarque de Holanda\\
13083-859\\
Cidade Universit\'aria, Campinas \\
Brazil}

\email{llohann@ime.unicamp.br}

\thanks{This work was financially supported by FAPESP, grant number 2009/07953-8.}


\date{}

\begin{abstract}Motivated to study the geometry of the exotic spheres constructed in \cite{s0}, we derive a necessary condition for non-negative sectional curvature in certain total spaces of Riemannian submersions with totally geodesic fibers. In particular, we derive a condition on the metric of the original base space  and prove that the bundles in \cite{s0} and \cite{DPR} have sections of negative curvature.\end{abstract}

\maketitle
\section{Introduction}
Recall from \cite{gw} that, for a Riemannian submersion with totally geodesic fibers $\pi:P\to N$, the vertizontal sectional curvatures of a plane $X\wedge U$ in $P$ is given by $sec_P(X,U)= |A^\dagger_XU|^2$, where $A^\dagger$ is the dual, with respect to $g_P$, of  the integrability tensor of $\pi$. In particular, a necessary condition for positive curvature on $P$ is the non-degenerancy of $A$. In this case, $\pi:P\to N$ is called a \emph{fat}-bundle. 

By the other hand, \cite{nara} proves that, there exists a universal bundle $V_k\to G_k$ with connection $\omega_0$ such that, if $\pi:(P,g_P)\to (M,g_M)$ is a $O(k)$ or $U(k)$ principal bundle with connection 1-form $\omega$, then, there is a bundle map $\Phi:P\to V_k$ such that $\Phi^*\omega_0=\omega$. In fact, $V_k\to G_k$ are exactly the respective Grassmanians. In particular, the universal principal $U(1)=S^1$-bundle is \emph{fat}.

In this work we assume that the degenerancy of the $A$-tensor in a bundle \linebreak $P'\to M$ is given in a controlled way  and prove a necessary condition for non-negative sectional curvature on $P'$. In fact, we prove the following

\begin{thm}\label{tg thm} Let $f^*P\to M$ be the pull-back of a fat Riemannian submersion with a connection metric with induced connection. Then, if $M$ is compact and $f^*P$ has non-negative sectional curvature, $f^{-1}(a)$ is a totally geodesic submanifold of $M$ for every regular point $a\in N$ of $f$.\end{thm}

At the end, we prove that, for any metric in $S^8$ or $S^{10}$, the bundles with connections as in Theorem 4.3 of \cite{s0}, whose total spaces admit isometric actions with exotic spheres as quotients, have sections of negative curvature. 

Here, we confound the terms submersion and bundle, when we discuss connection metrics, since they are known to be equivalent (\cite{gw}). We follow the notation and conventions of \cite{gw}.

\section{Graphics and Pull-back Bundles}\label{sec 1}
We start with a charecterization of the normal space of a graph and its second fundamental form. Then, we connect it to the study of the metric defined by the standard embedding of a pull-back bundle. Afterward, we show that connection metric with the pull-backed connection can be realized in this fashion, as far as $M$ is compact, then we proof Theorem \ref{tg lem}.

Let $(M,g_M)$ and $(N,g_N)$ be Riemannian manifolds and $f:M\to N$ be a smooth map. We consider \textit{the graph of $f$}, $\Gamma_f$, as the subset 
\[\Gamma_f=\{(x,y)\in M\times N~|~y=f(x)\}.\]
 Let $F:M\to \Gamma_f$ be the diffeomorphism defined by $F(x)=(x,f(x))$ and observe that, for $\nu\Gamma_f$, the normal space to $F$, $F^*(\nu \Gamma_f)=f^*TN$. An isomorphism is given by the map $(\Xi_N)_x:T_{f(x)}N\to \nu \Gamma_f$  defined by $\Xi_N(Y)=(-df^\dagger(Y),Y)$, where $g_M(df^\dagger(X),Y)=g_N(X,dfY)$ being $T(M\times N)=TM\times TN$ identified in the natural way. We also define the isomorphism $\Xi:TM\oplus f^*TN\to F^*T(M\times N)$ as $\Xi(X,Y)=dF(X)+\Xi_N(Y)$.

Writing the vector $(X,Y)\in T(M\times N)$ as a column, one can verify that
\begin{gather}\label{Xi} \Xi^{-1}\begin{pmatrix}X\\Y\end{pmatrix}=\begin{pmatrix} (1+df^\dagger df)^{-1}&df^\dagger(1+dfdf^\dagger)^{-1}\\-df(1+df^\dagger df)^{-1}&(1+dfdf^\dagger)^{-1}\end{pmatrix}\begin{pmatrix}
	X\\Y\end{pmatrix}.\end{gather}

In particular, for $pr_{\nu\Gamma_f}$, the orthogonal prjection to $\nu\Gamma_f$, we have
\begin{lem}\label{nu graph}For $(X,Y),(X',Y')\in T(M\times N)$, 
\[\label{proj nu}g_M\times g_N\left(pr_{\nu\Gamma_f}\begin{pmatrix}
X\\ Y\end{pmatrix},\begin{array}{cc}X'\\Y'\end{array}\right)=g_N((1+dfdf^\dagger)^{-1}(Y-df(X)),Y'-df(X')).\]\end{lem}
\begin{proof}The lemma follows by noticing that 
\[pr_{\nu\Gamma_f}=\Xi pr_{f^*TN} \Xi^{-1}=\begin{pmatrix}
0&	-df^\dagger\\0 & 1\end{pmatrix}\begin{pmatrix} 0&0\\-df(1+df^\dagger df)^{-1}&(1+dfdf^\dagger)^{-1}\end{pmatrix},\]
 and $df(1+df^\dagger df)^{-1}=(1+dfdf^\dagger)^{-1}df$.
\end{proof}

Let $\Pi(X,Y)=Y-df(X)\in f^*TN$ and $O=(1+dfdf^\dagger)^{-1}$, so $pr_{\nu\Gamma_f}=\Pi^\dagger O\Pi$. 

Let $\tilde X=(X,Y)$ and $\tilde X'=(X',Y')$ be two vector fields in $\Gamma_f$ arbitrarly extended to $M\times N$. Noticing that $Y=df(X)$ and $Y'=df(X')$, we have that, for $\overline{\nabla}$, $\nabla^M$ and $\nabla^M$, the covariant derivatives associated to $g_M\times g_N$, $g_M$ and $g_N$,
\[\overline{\nabla}_{\tilde X}\tilde X'=(\nabla^M_XX',\nabla^N_YY').\]
In particular, we can think of $\nabla^N_{dfX}dfX'=\nabla^N_YY'$ as a well-defined vector-field, being the second fundamental form, $\rm{II}_f$, of $\Gamma_f$ completly determined by
\begin{gather}\label{IIf}d^2f(X,X')=\Pi(\overline{\nabla}_{dF(X)}dF(X'))=\nabla^N_{dfX}dfX'-df(\nabla^M_XX').\end{gather}
In fact, 
\[\rm{II}_f(\tilde X,\tilde X')=pr_{\nu\Gamma_f}\overline{\nabla}_{\tilde X}\tilde X'=\Pi^\dagger O\Pi\bar\nabla_{\tilde X}\tilde X'=\Pi^\dagger Od^2(X,X').\]
It follows that $d^2f$ is a simetric $(2,1)$-tensor. 

Now, let $\pi:(P,g_P)\to (N,g_N)$ be a Riemannian submersion. We define \linebreak$\pi_f:f^*P\to M$, the \emph{pull-back of $\pi$} as the manifold $f^*P=\{(x,p)\in M\times P~|~f(x)=\pi(p)\}$ endowed with the restriction of the projection to the first coordinate. We write $\lr{,}$ for $g_M\times g_P$ and consider $f^*P$ as a submanifold of $M\times P$. We denote by $\cal H_p$ and $\cal V_p$ the horizontal and vertical subspaces of $T_pP$ and by $\cal L_p:T_{\pi(p)}N\to P$ the horizontal lift. We observe that
\begin{prop}Let $\tilde\pi=(\rm{id}\times \pi):M\times P\to M\times N$. Then, the restriction of $\tilde\pi$ to $f^*P\to \Gamma_f$ is a Riemannian submersion. Furthermore, the derivative of $\tilde\pi$ induces an isometry $\nu_{(x,p)}P\to\nu_{(x,f(x))}\Gamma_f$ for every $(x,p)\in f^ *P$.\end{prop}
\begin{proof}We first observe that 
\[T_{(x,p)}f^*P=\{(X,E)\in T_xM\times T_pP~|~df_x(X)=d\pi_p(E)\}.\]
 In particular, $V_{(x,p)}=\{(0,v)~|~v\in V_p\}$ and $\nu_{(x,p)} f^*P\subset T_xM\times \cal H_p$. Since $\rm{id}\times d\pi_p$ is an isometry when restricted to $T_xM\times \cal H_p$ and it sends \linebreak$(T_xM\times \cal H_p)\cap T_{(x,p)}f^*P$ to $T\Gamma_f$, it sends vectors orthogonal to $Tf^*P$ to $\nu \Gamma_f$ in an isometric way. 
Now, the first statement follows by noticing that  $tilde X=(X,\cal L_p(df_x(X)))\in T_{(x,p)}f^*P$ and is orthogonal to $\{0\}\times V_p=V_{(x,p)}$. But, 
\[||(X,\cal L_p(df_x(X)))||^2_{M\times P}=||X||^2_M+||df_x(X)||^2_N=||d\tilde\pi(X,\cal L_p df_x(X))||^2_{M\times N}.\]\end{proof}

For the rest of the paper, we assume that the fibers of $\pi$ are totally geodesic, and, first reduce the proof of Theorem \ref{tg thm} to the case of pull-back bundles. With effect, let $(f^*P,g_P')\to (M,g_M)$ be a submmersion with a connection metric induced by the connection in $\pi:P\to N$, with isometric fibers in both bundles. Then

\begin{prop} If $M$ is compact, then, there exists  a metric $g_M'$ on $M$, $\epsilon>0$ and an isometry $(f^*P,g_{P'})\to (f^*P,g_M'\times g''_P)$, where $g''_P$ is the connection metric associated with $\epsilon g_N$ and the connection in $\pi$. Furthermore, given $a\in N$, for any submanifold $L\subset f^{-1}(a)\subset M$, the second fundamental form of $L$ on the metrics $g_M$ and $g_M'$ coincide.\end{prop}
\begin{proof}Since connection metrics are specified by the connection and the metric in the base, according to what was discussed so far, for the first statement, it is sufficient to prove that there is a metric $g_M'$ and $\epsilon>0$ such that, the graph  of $f:(M,g_M')\to (N,\epsilon g_N)$ is isometric to $(M,g_M)$.  Write $df^\dagger$ for the dual of $df$ with respect to $g_M$ and $g_N$ and let $\epsilon>0$ be a real number such that $1-\epsilon dfdf^\dagger$ is still  positive definite. Then, for $g_M'=(1-\epsilon dfdf^\dagger)g_M$ we have that the metric in the graph of $f:(M,g_M')\to (N,\epsilon g_N)$ is
\[F^*(g_M'\times \epsilon g_N)=g_M'+\epsilon f^*g_N=g_M-\epsilon f^*g_N+\epsilon f^*g_N=g_M\]
as desired.
Now, let $X,Y$ be vector fields of $L$. We observe that, for $Z$, a vector field of $M$ restricted fo $L$ we have 
\begin{gather}\label{gM'}g_M'(X,Z)=g_M(X,Y)-\epsilon g_N(dfX,dfZ)=g_M(X,Y)\end{gather}
since the entire tangent bundle of $L$ is contained in the kernel of $df$. In particular, the normal bundles and the intrinsic metrics of $L$ coincide for both metrics. It proves the assertion, since the second fundamental form of a fixed submanifold just depends on the normal projection - with effect, recall that $\rm{II}(X,Y)=\frac{1}{2}[X,Y]$.\end{proof}

Proceeding to the second fundamental form of the pull-back bundle, we write $A(X,Y)=\frac{1}{2}pr_{\cal V}[pr_{\cal H}X,pr_{\cal H}Y]$, for the $A$-tensor of $\pi$ and $g_P(A^\dagger_XU,Y)=g_P(U,A(X,Y))$. We have
\begin{prop}\label{prop II} If $\tilde X=(X,Y),\tilde X'=(X',Y')$ are vector fields on $f^*P$, then 
\[d\tilde\pi\rm{II}(\tilde X,\tilde X')=\Pi^\dagger O(d^2f(X,X')+\Lambda(Y,Y')),\]
where $\Lambda(Y,Y')=(0,d\pi(-A_YY'-A_{Y'}Y))$.\end{prop}
\begin{proof}From \cite{gw}, we have
\[d\pi\nabla^P_{Y}Y'=\nabla^N_{d\pi Y}d\pi Y'-d\pi(A^\dagger_{Y'}Y+A^\dagger_{Y}{Y'}).\]
So, for $\tilde Z=(X'',Y'')$, a normal vector and $g_{M\times N}=g_M\times g_N$,
\begin{align*}\lr{\rm{II}(\tilde X,\tilde Y),\tilde Z}&=\lr{pr_{\nu f^*P}(\nabla^M_XX',\nabla^P_YY'),\tilde Z}
\\&=g_{M\times N}(pr_{\nu\Gamma_f}d\tilde\pi(\nabla^M_XX',\nabla^P_YY'),d\tilde\pi\tilde Z)
\\&=g_{M\times N}(\Pi^\dagger O\Pi(\nabla^M_XX',\nabla^N_{d\pi Y}d\pi Y'-d\pi(A^\dagger_{Y'}Y+A^\dagger_{Y}{Y'})),d\tilde\pi\tilde Z)
\\&=g_{M\times N}(\Pi^\dagger O(d^2f(X,X')+\Lambda(Y,Y')),d\tilde\pi\tilde Z).\qedhere\end{align*}\end{proof}

The following lemma culminates in Theorem \ref{tg thm}.

\begin{lem}\label{tg lem}If $f^*P\to M$ has non-negative curvature, then 
\[A(\cal LOd^2f(X,X),\cal Ldf Z)=0.\]\end{lem} 
\begin{proof} We recall that, given three vectors $X,Z,U$, for $R$, the $(4,0)$-curvature tensor, we have 
\[R(X,tU+Z,tU+Z,X)=t^2R(X,U,U,X)+2tR(X,U,Z,X)+R(X,Z,Z,X)\] 
In particular, if $R(X,U,U,X)=0$, $R(X,tU+Z,tU+Z,X)\geq 0$ for all $t$ if and only if $R(X,Z,ZX)\geq 0$ and 
\begin{gather}\label{RR}R(X,U,Z,X)=0.\end{gather} 

By the other hand, writing $R_M,~ R_N$ and $R_P$ for the $(4,0)$-curvature tensors of $M,~N$ and $P$, respectively, Gauss formula together with Proposition \ref{prop II} gives us
\begin{align*}R(\tilde X',\tilde X,\tilde X,\tilde Z)&= R_M(X',X,X,X'')+R_P(Y',Y,Y,Y'')\\ &+g_N(O(d^2f(X,X)-2A^*_YY),d^2f(X',X'')+\Lambda(Y',Y''))\\&-g_N(O(d^2f(X,X')+\Lambda(Y,Y')),d^2f(X,X'')+\Lambda(Y,Y''))\end{align*}

Where $\tilde X,\tilde X'$ and $\tilde Z$ are as in Proposition \ref{prop II}. Observing that, if $X\in\ker df$, $\tilde X=(X,0)$ is horizontal in $f^*P$, and taking $\tilde U=(0,U)$ a vertical vector and $\tilde Z=(Z,\cal L_pdfZ)$ an horizontal vector, we conclude that
\begin{gather}\label{R1}R(\tilde U,\tilde X,\tilde X,\tilde U)=0\end{gather}
and
\begin{align}\nonumber R(\tilde U,\tilde X,\tilde X,\tilde Z)&=g_N(Od^2f(X,X),d^2f(0,Z)+\Lambda(U,\cal L_pdfZ))\\
\label{R2} &=-g_P(A(\cal Ldf Z,\cal LOd^2f(X,X)),U)\end{align}
So, condition  \eqref{RR} holds for every $U$ if and only if $A(\cal Ldf Z,\cal LOd^2f(X,X))=0$. \end{proof}

\begin{proof}[Proof of Theorem \ref{tg thm}:] Now, for $\pi(p)=a\in N$ and $\pi$ fat, the map $\Xi_X:T_aN\to \cal V_p$ defined by
\[ Y\mapsto A(\cal L_pOd^2f(X,X),\cal L_pY)\]
is surjective  if and only if $d^2f(X,X)\neq 0$. In particular, $\dim\ker \Xi_X>\dim N-\dim \cal V$ if and only if $d^2f(X,X)=0$. By the other hand, condition \eqref{RR} implies that $df(T_xM)\subset \ker\Xi_X$, so, if $a$ is a regular point $\dim \ker\Xi_X\geq \dim df(T_xM)=\dim N> \dim N-\dim\cal V$.

Since, in a regular point, we always can choose an extension $\bar X$ of $X$ such that $df\bar X=0$, we have  \[d^2f(X,X)=\nabla^N_{df\bar X}df\bar X-df(\nabla^M_{\bar X}\bar X)=-df(pr_{(\ker df)^\bot}\nabla^M_{\bar X}\bar X)\]
We conclude the proof by observing that $pr_{(\ker df)^\bot}\nabla^M_{\bar X}\bar X$ is the second fundamental form of $f^{-1}(a)$ and that $df$ is injective on $(\ker df)^\bot$. \end{proof}

\section{An Application to the study of Exotic Spheres}

Given a pull-back bundle $f^*P\to M$ with the induced connection, Theorem \ref{tg thm} gives an obstruction for the problem  of finding a metric in $M$ such that the respective connection metric has non-negative curvature. I.e., for any such metric, the regular level sets of $f$ must be totally geodesic. Assuming then, that a metric satisfy this condition, we have 

\begin{lem}\label{tg fol lem} Let $c:[0,1]\to N$ is a curve such that $c(0)$ is a singular point and $c(t)$ is regular for all $t>0$. Then, the regular level sets $f^{-1}(c(t))$ are isometric to $f^{-1}(c(1))$ for $t>0$ and there is a totally geodesic embedding $f^{-1}(c(1))\hookrightarrow f^{-1}(c(0))$.\end{lem}

From this lemma, we conclude that singular level sets must have, at least the dimension of near fibers. In particular it implies that there is no metric on $S^{8}$ or $S^{10}$ such that the bundles $(-h\eta_8)^*S^7$ or $(-hb_{10})^*S^7$ have connection metrics with non-negative curvature, where $\eta:S^3\to S^2$ and $h:S^7\to S^4$ are the \emph{Hopf} bundles and $\eta_8$ and $b_{10}$ are analytic suspensions of $\eta$ and a generator of $\pi_6S^3$ given in \cite{s1}. By the other hand, these bundles admit $S^3$ actions (which are isometric for many metrics and connections) with quotients exotic spheres. 

\begin{proof}[Proof of Lemma \ref{tg fol lem}:] Write $L=f^{-1}(c(1))$ and $S=f^{-1}(c(0))$. From \cite{jw1}, we now that the curve $c$ induces isometries $\phi_t:f^{-1}(c(1))\to f^{-1}(c(t))$. So we have an embedding $i:(0,1]\times L\to M$. Now, we can find an embeddin $\tilde i:[0,1]\times L\to M$ by taking $\tilde i(0,x)=\lim_{t\to 0}i(t,x)$. it defines a continuos bijective map $\tilde i_0:L\to M$, smoothnes is guaranteed by extending the vector fields $X(t)=di_{(t,x)}(0,X)$ and geodesics doesn't go out of $L$ since both exponential and distance maps are continuos and
\[0=\lim_{t\to 0}(d(\gamma_{X(t)}(s),S))=d(\gamma_{X(0)}(s),S),\]
being $\gamma_X(s)$ the geodesic of the vector $X$ at instant $s$. 
\end{proof}

We also would like to remark that, fot the \emph{Hopf}-bundles $\eta,~h$ and $H:S^{15}\to S^8$, the proof of Theorem \ref{tg thm} also demonstrates that:

\begin{thm} Let  $f^*P\to M$ be a submersion induced by one of the \emph{Hopf}-bundles above with the induced connection. Then, if $\rm{rank} df_x>1$ for all $x$ and $f^*P$ has non-negative sectional curvature, then, any level set of $f$ is totally geodesic.\end{thm}

An example of map that violates this condition is the geodesic two fold $\rho_2:S^n\to S^n$ starting at $e_0\in S^n$, defined by
\[\cos te_0+\sin tX\mapsto \cos 2te_0+\sin 2tX,\]
where $X\in T_{e_0}S^n$. In the general case of a geodesic $k$-fold, the condition given by Lema \ref{tg lem} can also be verified. We believe that, in certain sense, these are the only maps in the case where $N$ is a round sphere, in the sense that, if an analytic map gives a connection metric with non-negative curvature, that it is a composition of a (non-Riemannian) submersion with a map like this.

\subsection*{Acknowledgment}
This work is part of the author's Ph.D. thesis in Unicamp, Brazil, under C. Dur\'an and A. Rigas.

\bibliography{bib1}{}

\begin{thebibliography}{1}

\bibitem{DPR}
Puettmann T. Rigas~A. Dur\'an, C.
\newblock An infinite family of gromoll-meyer spheres.
\newblock {\em Archiv der Mathematik}, 95:269--282, 2010.

\bibitem{gw}
Walshap~G. Gromoll, D.
\newblock {\em Metric Foliations and Curvature}.
\newblock Birkhäuser Verlag, Basel, 2009.

\bibitem{jw1}
D.~Johnson and L.~Whitt.
\newblock Totally geodesic foliations.
\newblock {\em J. Diff. Geo.}, 15:225--235, 1980.

\bibitem{nara}
S.~Narasimham and S.~Ramanan.
\newblock Existence of universal connections.
\newblock {\em American journal of Mathematics}, 83:563--572, 1961.

\bibitem{s0}
L.~Speran\c{c}a.
\newblock On explicit constructions of exotic spheres.
\newblock {\em preprint}, arXiv, 2012.

\end{thebibliography}
\bibliographystyle{plain}

\end{document}